\documentclass[11pt]{amsart}
\usepackage{amssymb, latexsym, amsthm}
\usepackage[all]{xy}
\usepackage{verbatim}
\usepackage{color}
\usepackage{amsmath}
\usepackage{txfonts}

\newcommand\longleftmapsto{\longleftarrow\joinrel\mapstochar}

\newtheorem{thm}{Theorem}[section] 

\newtheorem{cor}[thm]{Corollary}

\newtheorem{prop}[thm]{Proposition}

\theoremstyle{definition}
\newtheorem{rem}[thm]{Remark}

\DeclareMathOperator{\mychar}{char} %

\newcommand\operA[2]{{\if!#2!\operatorname{#1}\else{\operatorname{#1}_{#2}^{\phantom{I}}}\fi}} 
\newcommand\set[1]{\{#1\}}

\newcommand\eq[1]{{(\ref{#1})}}
\newcommand\Pref[1]{{Proposition~\ref{#1}}}%
\newcommand\Cref[1]{{Corollary~\ref{#1}}}%
\newcommand\Rref[1]{{Remark~\ref{#1}}}%
\newcommand\Tref[1]{{Theorem~\ref{#1}}}%
\newcommand\Sref[1]{{Section~\ref{#1}}}%

\newcommand\tensor[1][]{{\otimes_{#1}}}
\def\sub{\subseteq}

\def\tr{{\operatorname{tr}}}

\def\disc{{\operatorname{disc}}}

\def\s{\sigma}

\newcommand\res[1][{}] {{\operatorname{res}_{#1}}}

\newcommand{\Trace}[1][]{\if!#1!\operatorname{Tr}\else{\operatorname{Tr}_{#1}^{\phantom{I}}}\fi} 

\long\def\forget#1\forgotten{{}} %
\newcommand\suchthat{{\,:\ \,}}

\def\co{{\,{:}\,}}
\newcommand\lam{{\lambda}}
\newcommand\etale{{\'{e}tale}}

\newcommand\Zent{{\operatorname{Z}}}
\def\({\left(}
\def\){\right)}

\newcommand\isom{{\,\cong\,}}

\def\Cass{{\operatorname{C}}}
\def\Calt{{\operatorname{C}^{\operatorname{alt}}}}
\def\Caltf{{\hat{C}^{\operatorname{alt}}}}

\newcommand{\sgn}[1]{{\operatorname{sgn}(#1)}}
\newcommand\Qf[1]{{\left<#1\right>}}
\newcommand\Pf[1]{{\left<\!\left<#1\right>\!\right>}}

\renewcommand\H[2]{{\operatorname{H}^{#1}\!\!\;({#2})}}
\newcommand\GP[1]{{\operatorname{GP}_{#1}}}
\renewcommand\P[1]{{\operatorname{P}_{#1}}}
\def\Quad{{\operatorname{Quad}}}

\newcommand\abc{{\alpha_1\alpha_2\alpha_3}}
\newcommand\ourphi{{\Qf{\alpha_1,\alpha_2,\alpha_3}}}
\newcommand\ourw{{\mathfrak{d}}}

\newif\iffurther
\furtherfalse

\newif\ifXY 
\XYtrue     
\ifXY

\fi 

\newcommand\sg[1]{{\left<#1\right>}}

\begin{document}


\title{The Alternative Clifford Algebra of a Ternary Quadratic Form}

\author{Adam Chapman}
\email{adam1chapman@yahoo.com}
\address{Department of Computer Science, Tel-Hai College, Upper Galilee, 12208 Israel}
\author{Uzi Vishne}
\email{vishne@math.biu.ac.il}
\address{Department of Mathematics, Bar-Ilan University, Ramat Gan, 55200 Israel}
\thanks{The second named author was partially supported by an Israel Science Foundation grant \#1623/16.}

\maketitle

\def\ra{{\rightarrow}}

\begin{abstract}
We prove that the alternative Clifford algebra of a nondegenerate ternary quadratic form is an octonion ring 
whose center is the ring of polynomials 
in one variable over the field of definition. 
\end{abstract}


\section{Introduction}

The Clifford algebra of a quadratic form is both an important invariant of the form and an interesting object in its own right. We propose a detailed study of the alternative Clifford algebra, defined in the same manner as the classical Clifford algebra, but in the realm of alternative algebras.

Let $F$ be a field of arbitrary characteristic and let $\varphi \co V \ra F$ be a quadratic form on an $n$-dimensional vector space~$V$ over~$F$.
The classical Clifford algebra $\Cass(\varphi)$ is defined as the tensor algebra of~$V$ modulo the relations $v^2 = \varphi(v)$ for every $v \in V$.
Its structure is well understood:~$\Cass(\varphi)$ is a simple algebra of dimension~$2^n$ over~$F$, whose center is~$F$ when~$n$ is even, and a quadratic \etale\ extension, corresponding to the discriminant~$\disc(\varphi)$, when $n$ is odd. The natural involution is defined by mapping $v \mapsto -v$ for all $v \in V$.
The algebra thus  has order dividing~$2$ in the Brauer group of its center. Indeed, the Clifford algebra is the second cohomological invariant of quadratic forms (see \cite[Section~14]{EKM}).

An alternative analogue was recently introduced by Musgrave \cite{Musgrave:2015}, who defined the {\bf{alternative Clifford algebra}} $\Calt(\varphi)$ as the
alternative tensor algebra of~$V$ modulo the same relations $v^2 = \varphi(v)$. Explicitly, the alternative Clifford algebra is the
free alternative algebra generated by $x_1,\dots,x_n$ over~$F$, modulo the relations
$$(u_1 x_1+\dots+u_n x_n)^2=\varphi(u_1,\dots,u_n)$$
for every choice of $u_1,\dots,u_n \in F$.
(Chapter~13 of~\cite{ZSSS} is devoted to the free alternative algebra.)
The main motivation is to study, through an algebraic construction, the higher cohomological invariants of quadratic forms (see \cite[Section 15 and 16]{EKM} and \cite{Voevodsky:2003}). When $n \leq 2$ the alternative Clifford algebra is identical to the associative one. Musgrave asks if the alternative Clifford algebra is finite dimensional, especially when $n=3$ (see \cite[Questions~3.1 and~3.4]{Musgrave:2015}).

In this paper we provide a complete description of the alternative Clifford algebra $\Calt(\varphi)$ for nondegenerate ternary forms.
\begin{thm}\label{Main}
Let $(V,\varphi)$ be a nondegenrate ternary quadratic space over a field~$F$. The alternative Clifford algebra~$\Calt(\varphi)$ is an octonion ring whose center is the ring of polynomials in one variable over~$F$.
\end{thm}

See \Tref{Main+} and its corollaries for an explicit description of the algebra.
The proof revolves around the ``alternative discriminant'', studied in \Sref{sec:disc}, which turns out to be a transcendental element in the center of the alternative algebra.
Following the description of the alternative Clifford algebra in \Sref{sec:main}, we observe that $\Calt(\varphi)$ has finitely many associative simple quotients, and consequently (\Cref{localize}):
\begin{cor}
$\Calt(\varphi)$ has a central localization which is
``alternative Azumaya'' in the sense that its simple quotients are all octonion algebras.
\end{cor}
In \Sref{sec:ci} we discuss the alternative Clifford algebra in the context of cohomological invariants.

\section*{Acknowledgements}
We thank Ivan Shestakov and for an insightful discussion and for
his careful comments on an earlier version of this paper, and an anonymous referee for suggesting significant shortcuts in the proof of the main result.

\medskip


\section{Background}\label{sec:back}

We briefly provide necessary background on alternative algebras.

\subsection{Nonassociative rings and algebras}

Let $A$ be a nonassociative ring. The {\bf{commutator}} of $x,y \in A$ is the element $[x,y] = xy - yx$. The {\bf{associator}} of $x,y,z \in A$ is the element $(x,y,z) = (xy)z - x(yz)$.
The {\bf{nucleus}} (also the ``associative center'') of $A$ is the set
$$N(A) = \set{a \in A \suchthat (a,A,A) = (A,a,A) = (A,A,a) = 0},$$
and the {\bf{center}} is
$$Z(A) = \set{a \in N(A) \suchthat [a,A] = 0}.$$
The nucleus is an associative subring of $A$, and the center is a commutative subring of the nucleus.

\subsection{Alternative rings}

A nonassociative ring~$A$ is {\bf{alternative}} if  $$(a,a,b) = (a,b,b) = 0$$
for every $a,b \in A$.
Linearizing, it follows that the associator alternates:
$$(a_{\s (1)},a_{\s (2)}, a_{\s (3)}) = \sgn{\s} (a_1,a_2,a_3)$$
for any $a_1,a_2,a_3 \in A$ and $\s \in S_3$. Any two elements of an alternative ring generate an associative subring (Artin, \cite[Chapter 2, Theorem~2]{ZSSS}). 

A~simple alternative ring is either associative or an octonion algebra over its center, which is a field \cite[Chapter 7, Section 3, Corollary 1]{ZSSS}. A prime alternative ring~$A$ is either associative or an octonion ring, provided that $3A \neq 0$ \cite[Chapter 9, Corollary to Theorem~9]{ZSSS}.

We list some useful identities. Let $L_b,R_b \co A \rightarrow A$ denote the multiplication maps $L_b(c) = bc = R_c(b)$. For any $b,c \in A$, we denote $b \circ c = bc + cb$. Rather than composition, we use the same notation for operators: $R_b \circ R_c = R_bR_c+R_cR_b$.
\begin{rem}\label{class}
Let $b, c\in A$. Then $R_b \circ R_c = R_{b \circ c}$ and
$L_b \circ L_c = L_{b \circ c}$.
\end{rem}

\begin{cor}\label{RR}
If $b \circ c = 0$ then $R_b,R_c$ anticommute and $L_b,L_c$ anticommute.
\end{cor}

For every $a,b$ in an alternative ring we have that
\begin{equation}\label{clear}
{}[a\circ b, a] = [b,a^2],
\end{equation}
because the subring generated by $a,b$  is associative.

An alternative algebra satisfies the Moufang identities \cite[Chapter 2, Lemma~7]{ZSSS}, which imply the identities
\begin{eqnarray}
(a,[b,a],c) & = & [a,(a,b,c)] \label{[a(abc)]}, \\
a \circ (a,b,c) & = & (a^2,b,c)\label{ao(abc)},
\end{eqnarray}
and \begin{equation}
(a,a\circ b,c) = (a^2,b,c)\label{(a,aob,c)},
\end{equation}
see \cite[Chapter 2, Equations (16)--(17$'$)]{ZSSS}.

\begin{rem}\label{BK}
Let $S$ be a set of generators for an alternative algebra. An element $u$ is central in the algebra if and only if $[u,S] = (u,S,S) = 0$
(Bruck-Kleinfeld's lemma, \cite[Chapter 13, Lemma~16]{ZSSS}). Similarly, $u$ is in the nucleus if and only if $(u,S,S) = (u,S,SS) = (u,SS,SS) = 0$
\cite[Lemma~2.4]{Sh-Zh}.
\end{rem}

Of some relevance is \cite{Sh-Zh}, which fully describes the free alternative superalgebra in one odd generator. By definition, an odd generator anti-commutes with itself. Shestakov and Zhukavets found a correspondence between this algebra and the skew part of the alternative Grassmann algebra, which can be viewed as an epimorphic image of our alternative Clifford algebra for the zero quadratic form (in any dimension). We do not give further details here, because the dimension of the alternative Grassmann algebra on any finite number of generators is finite, hence too small to be directly helpful for our cause.

\section{The alternative discriminant}\label{sec:disc}

\newcommand\bil{{\varphi}} 
\newcommand\bilalone{{\varphi(\cdot,\cdot)}} 

Let $\varphi \co V \ra F$ be a nondegenerate ternary quadratic form, and write $\bilalone$ for the underlying symmetric bilinear form given by $\bil(v,w)=\varphi(v+w)-\varphi(v)-\varphi(w)$. We consider the generating subspace $V \sub \Calt(V,\varphi)$.
Since $u^2 = \varphi(u)$ for every $u \in V$, we have for $u,v \in V$ that $u\circ v=\bil(u,v)$ and $u\circ u=2\varphi(u)=\bil(u,u)$.
For $u,v,w \in V$, let $$\delta(u,v,w) = u\circ (vw)-\bil(v,w) u = u(vw)-(wv)u.$$
\begin{prop}
The form $\delta$ alternates.
\end{prop}
\begin{proof}
We have that
$$\delta(u,v,v)=2\varphi(v)u-\bil(v,v) u=0$$ and
\begin{eqnarray*}
\delta(u,u,w) & = & u \circ (uw)-\bil(u,w)u \\
& = & \varphi(u)w+uwu-\bil(u,w)u=\varphi(u)w-(wu)u=0.
\end{eqnarray*}
\end{proof}
Consequently, the one-dimensional space $F \delta(u,v,w)$ is independent of the choice of basis $u,v,w$ of~$V$. Notice that we can always choose a basis for which $v,w$ are orthogonal with respect to $\bilalone$, in which case $\delta(u,v,w) = u \circ (vw)$.

\begin{prop}\label{herecent}
The element $\delta(u,v,w)$ is central in $\Calt(V,\varphi)$.
\end{prop}
\begin{proof}
By \Rref{BK} one only has to verify that $\delta(u,v,w)$ commutes and associates with elements in $V$, which follows from:
$$[u, \delta(u,v,w)] = [u,u \circ (vw)]  - \bil(v,w)[u,u] = [u^2,vw] = 0$$
and
$$(u,v,\delta(u,v,w))=(u,v,u\circ(vw))=(u,v,vw)\circ u=(u,v\circ u,vw)=0.$$
\end{proof}

\section{The alternative Clifford algebra of a ternary form}\label{sec:main}

Let $(V,\varphi)$ be a nondegenerate 3-dimensional quadratic space. Let $V_0$ be a nondegenerate 2-dimensional subspace of $V$. Then $Q = \Calt(V_0,\varphi|_{V_0})$ is a quaternion algebra over $F$. Fix a basis $\set{u,v}$ of $V_0$, and let $d_0 = 4 \varphi(u)\varphi(v) - \bil(u,v)^2$ be the determinant of $\varphi|_{V_0}$ with respect to this basis, which is nonzero by assumption. Since $u \circ v = \bil(u,v)$, we have that
\begin{equation}\label{uv2}
{}[u,v]^2 = (uv+vu)^2 - 2(uvvu+vuuv) = \bil(u,v)^2 - 4 \varphi(u)\varphi(v) = -d_0.
\end{equation}

Let $w \in V$ be orthogonal to $V_0$, so that $V = V_0 \perp F w$. It follows that $u \circ w = v \circ w = 0$. Let $\ourw = \delta(u,v,w)$, which is central by \Pref{herecent}.

An octonion algebra is by definition the Cayley-Dickson double of a quaternion algebra over a field.
A ring~$R$ which is a faithful module over an integral domain~$C \sub \Zent(R)$ is an octonion ring if $(C -\set{0})^{-1}R$ is an octonion algebra. 
Octonion rings can be recognized in the following manner.
\begin{rem}\label{fadicha2}  
Let~$R$ be an alternative ring which is a faithful module over an integral domain~$C$. Let $Q \sub R$ be a quaternion subring, with the reduced trace $\tr \co Q \ra C$.
If~$z \in R$ satisfies
\begin{equation}\label{za}
z \circ a = \tr(a)z
\end{equation} for every $a \in Q$ and $z^2 = \gamma \in C$ is nonzero, then the subalgebra generated by $Q$ and $z$ in $R$ is the octonion algebra $(Q,\gamma)_C$.
\end{rem}

We now prove \Tref{Main} in the following form:
\begin{thm}\label{Main+}
The alternative Clifford algebra $\Calt(V,\varphi)$ is the octonion ring
$$(Q,\ \ourw^2+ d_0 \varphi(w))_{F[\ourw]},$$
and $F[\ourw]$ is a polynomial ring over $F$.
\end{thm}
\begin{rem}
The sum $\ourw^2+d_0 \varphi(w)$ is well defined up to squares. Indeed, in characteristic not $2$, the product $d_0 \varphi(w)$ is the determinant of $\varphi$, which is well defined up to squares. In characteristic $2$,  $d_0$ is a square and $\varphi(w)$ is well defined up to squares.
\end{rem}
\begin{proof}[Proof of \Tref{Main+}]
Let $z = [w, uv]$. First we note that by \Cref{RR},
$$(u,v,w) + [w,uv] =  - u(vw) + w(uv) = -u(w \circ v) = 0,$$
so that $z = -(u,v,w)$ vanishes in the associative Clifford algebra $\Cass(V,\varphi)$.
Next, we claim that $\Calt(V,\varphi)$ is generated over $F$ by $u,v,\ourw$ and $z$.
Indeed,
$$z + \ourw 
 = 2w(uv) - (u \circ v)w = 2w(uv) - \bil(u,v)w,$$
and therefore by \eq{uv2}
$$
(z+\ourw)[u,v] = (2 w(uv)-w\bil(u,v))[u,v]=w[u,v]^2 = -d_0 w;
$$
so $w \in F[\ourw,u,v,z]$.
In other words, over $F[\ourw]$, $\Calt(V,\varphi)$ is generated by $Q = F[u,v]$ and $z$.
We now show that \eq{za} holds for $a \in Q$. It suffices to prove this claim for $a = u,v,uv$, and indeed, by \eq{ao(abc)},
$$z \circ u = -(u,v,w) \circ u = -(u^2,v,w) = 0;$$
$$z \circ v = (v,u,w) \circ v = (v^2,u,w) = 0;$$
and
$$z \circ (uv) = [w,uv] \circ (uv) = [w, (uv)^2] = [w, \bil(u,v)uv - u^2v^2]  = \bil(u,v)z$$
which is equal to $\tr(uv)z$ because $\tr(uv) = uv + vu = u \circ v = \bil(u,v)$.

As a final preparation, we have that $z \circ w = [w^2,uv] = 0$ and
\begin{equation}\label{xx6}
w(uv) - \ourw = w(uv) - \delta(w,u,v) = (vu)w = \bil(u,v)w - (uv)w.
\end{equation}
Now let us compute
\begin{eqnarray*}
z^2+\bil(u,v)^2\varphi(w) & =& ([w,uv]+\bil(u,v)w)^2\\
& = & (2w(uv)-\ourw)^2\\
& = & 4(w(uv))^2-4\ourw(w(uv))+\ourw^2\\
& \stackrel{\eq{xx6}}{=}& \ourw^2-4((uv)w)(w(uv))+4\bil(u,v)\varphi(w)uv\\
& =& \ourw^2-4\varphi(w)(uv)^2+4\bil(u,v)\varphi(w)uv \\
& = & \ourw^2+4\varphi(u)\varphi(v)\varphi(w);
\end{eqnarray*}
so that $z^2 = \ourw^2 + d_0\varphi(w)$. By \Rref{fadicha2} $\Calt(V,\varphi) \isom (Q,z^2)_{F[\ourw]}$.
(Note that by \cite[Chapter 13, Theorem~12]{ZSSS}, the square of any associator in an alternative algebra generated by three elements is central, so $z^2 \in F[\ourw]$ was expected).

To show that $\ourw$ is transcendental over~$F$, reverse the argument. The octonion ring $C' = (Q,\lam^2+d_0 \varphi(w))_{F[\lam]}$ is generated by a copy of $Q$ and an element $z'$ satisfying \eq{za} and ${z'}^2 = \lam^2+d_0 \varphi(w)$. Let $w' = -\frac{1}{d_0}(z'+\lam)[u,v]$. Then $(t_1u+t_2v+t_3w')^2 = \varphi(t_1 u+t_2 v+ t_3w')$ because $w' \circ u = w' \circ v = 0$, showing that the map $\Calt(\varphi) \ra C'$ preserving $V = Fu+Fv+Fw$ and defined by $z \mapsto z'$ is well defined, and being onto, it is an isomorphism.
\end{proof}

\begin{cor}
The subring $F[\ourw]$ is both the center and the nucleus of $\Calt(\varphi)$.
\end{cor}

\begin{cor}\label{mainFnot2}
Let $F$ be a field of characteristic not $2$. A nondegenerate ternary form can be presented diagonally as $\varphi = \Qf{\alpha_1,\alpha_2,\alpha_3}$, and then $Q = (\alpha_1,\alpha_2)_F$ and
$$\Calt(\varphi) \isom (\alpha_1, \alpha_2, \ \ourw^2+ 4 \alpha_1\alpha_2\alpha_3)_{F[\ourw]}.$$
\end{cor}

\begin{cor}\label{mainF2}
Let $F$ be a field of characteristic~$2$. A nondegenerate ternary form can be presented as $\varphi = [\alpha_1,\alpha_2] \perp \Qf{\alpha_3}$ (so that
$\varphi(t_1,t_2,t_3) = \alpha_1 t_1^2 + t_1 t_2 + \alpha_2 t_2^2 + \alpha_3 t_3^2$)
and then $Q = [\alpha_1\alpha_2,\alpha_2)_F$, $d_0 = 1$, and
$$\Calt(\varphi) \isom [\alpha_1\alpha_2, \alpha_2, \ \ourw^2+ \alpha_3)_{F[\ourw]},$$
following the notation of \cite{ElduqueVilla:2005} for Hurwitz algebras in characteristic 2.
\end{cor}

\section{Simple quotients of the alternative Clifford algebra}\label{S7}

Recall that the associative Clifford algebra is either simple or the direct sum of two isomorphic quaternion algebras over~$F$, depending on the discriminant.
We now describe the simple quotients of $\Calt(\varphi)$, following the notation of \Tref{Main+}.
\begin{prop}\label{simplequot}
Let $\bar{F}$ be the algebraic closure of $F$. The simple quotients of $\Calt(\varphi)$ are \begin{enumerate}\item The octonion algebras $(Q,\theta^2 + d_0\varphi(w))_{F[\theta]}$ for every $\theta \in \bar{F}$ such that $\theta^2  + d_0 \varphi(w) \neq 0$;
\item The (simple quotients of the) associative Clifford algebra.
\end{enumerate}
\end{prop}
\begin{proof}
The image of the center in a simple quotient is a field extension of~$F$. Suppose $\ourw \mapsto \theta$ for $\theta \in \bar{F}$.
If $\theta^2+d_0 \varphi(w) \neq 0$ then the image is the octonion algebra $(Q,\theta^2+d_0 \varphi(w))_{F[\theta]}$. Otherwise, $z^2$ maps to zero, making the image of $Qz$ equal to the radical in the quotient, so by simplicity $z$ maps to zero, and the whole algebra maps to the quaternion algebra $Q_{F[\theta]}$.
 If $\theta \not \in F$, this is the associative Clifford algebra.
\end{proof}

For example, in characteristic not $2$, mapping $\ourw \mapsto 0$ we obtain a projection onto the octonion algebra $\Calt(\Qf{\alpha_1,\alpha_2,\alpha_3}) \ra (\alpha_1,\alpha_2,\alpha_3)_F$. (This was observed in \cite[Prop.~2.17]{Musgrave:2015}.)

\forget
\begin{rem}
The octonion algebra is endowed with a natural involution $v \mapsto -v$ for $v \in V$, inducing the symplectic involution of $Q$. Under this involution $\ourw$ is symmetric, and~$z$~is antisymmetric. Since $x_1x_2$ is antisymmetric, we get that $x_3$ is antisymmetric as well. This involution of $\Calt(V,\varphi)$, which can be defined by $\bar{v} = - v$ for all $v \in V$, projects to the symplectic involution of the associative Clifford algebra $\Cass(V,\varphi)$.
\end{rem}
\forgotten

\medskip

\begin{cor}\label{localize}
The localization of $\Calt(\varphi)$ at the central element $(u,v,w)^2  = \ourw^2+d_0 \varphi(w)$ is ``alternative Azumaya'' in the sense that its simple quotients are all octonion algebras.
\end{cor}

On the other hand, let $\sg{z}$ denote the ideal of $\Calt(\varphi)$ generated by the associator $z = (u,v,w)$.
\begin{rem}\label{a=0}
The algebra $\Calt(\varphi)/\sg{z}$ is associative.

\end{rem}
This is a consequence of Moufang's theorem \cite[Appendix]{Smiley}, and coincides with the case $\ourw^2 + d_0 \varphi(w) \mapsto 0$ of \Pref{simplequot}.

\section{Central fractions of the alternative Clifford algebra}\label{sec:ci}

For a nondegenerate ternary quadratic form $\varphi$,
let $\Caltf(\varphi)$ denote the algebra of central fractions of the alternative Clifford algebra $\Calt(\varphi)$, namely $\Caltf(\varphi) = F(\ourw) \tensor_{F[\ourw]} \Calt(\varphi)$ where $\ourw = \delta(u,v,w)$ for some basis $u,v,w$ of~$V$. More explicitly, assuming in this section that $\mychar F \neq 2$ and writing $\varphi = \Qf{\alpha_1,\alpha_2,\alpha_3}$, we have by \Cref{mainFnot2} that
$$\Caltf(\varphi) \isom (\alpha_1,\alpha_2,\lam^2+4\abc)_{F(\lam)}.$$

Assuming in this section that $\mychar F \neq 2$, we now put this algebra in context, by comparing it to the cohomological invariants of quadratic forms.

\subsection{Cohomological invariants}

Recall from \cite[Theorem~17.3]{GMS} that the cohomological invariants over~$F$ of ternary quadratic forms are freely spanned over the cohomology ring~$\H{*}{F,\mu_2}$ by the Steifel-Whitney invariants $\omega_i \in \H{i}{F,\mu_2}$, $i = 0,\dots,3$, where for $\varphi = \Qf{\alpha_1,\alpha_2,\alpha_3}$ the invariants are defined by $\omega_0(\varphi) = 1$, $\omega_1(\varphi) = (\abc)$, $\omega_2(\varphi) = (\alpha_1,\alpha_2)+(\alpha_1,\alpha_3)+(\alpha_1,\alpha_3)$ and $\omega_3(\varphi) = (\alpha_1,\alpha_2,\alpha_3)$. For example, the associative Clifford algebra $\Cass(\varphi)$ is the quaternion algebra $(\alpha_i,\alpha_j)$ over~$F[\sqrt{\disc(\varphi)}]$ for any $i \neq j$. Summing up in $\H{2}{F,\mu_2}$, this invariant corresponds to $\res_{F[\sqrt{-\omega_1(\varphi)}]/F}(\omega_2)$. Computation shows that $\omega_1 \cup \omega_2 = \omega_3 + (-1)\cup \omega_2$.

\subsection{The associated Pfister form}\label{ss:Pfister}

Recall that an $m$-fold Pfister form is a tensor product $\Pf{\gamma_1,\dots,\gamma_m} = \Pf{\gamma_1}\tensor \cdots \tensor \Pf{\gamma_m}$, where
$\Pf{\gamma} = \Qf{1,-\gamma}$. If $\pi$ is a Pfister form, $\pi'$ denotes the pure subform.
The norm form of $\Caltf(\Qf{\alpha_1,\alpha_2,\alpha_3})$ is the Pfister form $\Pf{\alpha_1,\,\alpha_2,\,\lam^2+4\abc}$ over~$F(\lam)$, corresponding to the symbol
$$\mu(\Qf{\alpha_1,\alpha_2,\alpha_3}) = (\alpha_1,\alpha_2,\lam^2 + 4 \abc)$$
in $\H{3}{F(\lam),\mu_2}$. This element, the octonion algebra over $F(\lam)$ and its norm form, determine each other up to isomorphism.

By symmetry we may write $\mu(\varphi) = (\alpha_i,\alpha_j,\lam^2+4\abc)$ for any $i \neq j$. Summing up the three presentations,
\begin{cor}
We have that $\mu(\varphi) = 3 \mu(\varphi) = (\lam^2+4\abc)\cup \omega_2(\varphi)$.
\end{cor}

\subsection{Presentation by Pfister forms}

It is of some interest to present $\mu(\varphi)$ in a form which is visibly symmetric.
Let $\Quad_k$ and $\P{k}$ denote the sets of $k$-dimensional forms and $k$-folds Pfister forms over $F$ up to isomorphism, respectively.
\begin{rem}\label{Q3=Q1P2}
There is a one-to-one correspondence
$$\Quad_3 \longleftrightarrow \Quad_1 \times \P{2}$$
by $\varphi = \Qf{a,b,c} \longmapsto (\Qf{abc},\Pf{-ab,-ac}) = (\Qf{\delta},\omega)$ and $\Qf{\delta}\omega' \longleftmapsto (\Qf{\delta},\omega)$.
Thus~$\varphi$ is isotropic iff $\omega$ is hyperbolic.

\forget 
Let $X \sub \Quad_1 \times \GP{2}$ be defined as $X = \set{(\Qf{\alpha},q) \suchthat \alpha \in D(q)}$.
There are one-to-one correspondences
$$\Quad_3 \longleftrightarrow X \longleftrightarrow \Quad_1 \times \P{2};$$
by
$$\Qf{a,b,c} \longleftrightarrow (\Qf{abc},\Qf{a,b,c,abc}) \longleftrightarrow (\Qf{abc},\Pf{-ab,-ac}),$$$$\Qf{\delta}\omega' \longleftrightarrow (\Qf{\delta},\Qf{\delta}\omega) \longleftrightarrow (\Qf{\delta},\omega).$$

The first map takes $\Qf{a,b,c} \mapsto (\Qf{abc},\Qf{a,b,c,abc})$, and is reversed by mapping $(\Qf{\alpha},q)$ to the orthogonal complement of $\Qf{\alpha}$ in $q$.
The second map is induced by the duality $(\Qf{\alpha},q) \mapsto (\Qf{\alpha},\Qf{\alpha} q)$, using the fact that if $\omega$ is similar to a Pfister form and $1 \in D(\omega)$ then $\omega$ is a Pfister form. In fact the maps can also be described as
\forgotten
\end{rem}

\begin{prop}\label{symmP}
The norm form of $\Caltf(\ourphi)$ is
$$\Pf{\lam^2+4\abc}\tensor \Pf{-\alpha_i\alpha_j,-\alpha_i\alpha_k}.$$
\end{prop}
Namely, if $\varphi = \ourphi$ corresponds to $(\Qf{\delta},\omega)$, then the norm form of $\Caltf(\varphi)$ is $\Pf{\lam^2+4\delta}_{F(\lam)}\tensor \omega$.
\begin{proof}
Compute, using that $\Qf{1,-(\lam^2+4\abc)}$ and $\Qf{-\abc,\frac{\lam^2+4\abc}{\abc}}$ are isomorphic over $F(\lam)$:
\begin{eqnarray*}
\mu(\ourphi) & =& \Qf{1,-\alpha_1,\,-\alpha_2,\,\alpha_1\alpha_2} \tensor \Pf{\lam^2+4\abc} \\
& = & \Qf{1,-\alpha_1,\,-\alpha_2,\,-\alpha_3} \tensor \Pf{\lam^2+4\abc} \\
& = & \Qf{1,-(\lam^2+4\abc), -\alpha_1,\,-\alpha_2,\,-\alpha_3} \\ && \quad \perp \Qf{-(\lam^2+4\abc)}\Qf{-\alpha_1,\,-\alpha_2,\,-\alpha_3} \\
& = & \Qf{-\alpha_1,\,-\alpha_2,\,-\alpha_3,-\abc,\frac{\lam^2+4\abc}{\abc}} \\ && \quad\perp \Qf{-(\lam^2+4\abc)}\Qf{-\alpha_1,\,-\alpha_2,\,-\alpha_3} \\
& = & \Qf{-\alpha_1,\,-\alpha_2,\,-\alpha_3,-\abc} \\ && \quad\perp \Qf{-(\lam^2+4\abc)}\Qf{-\alpha_1,\,-\alpha_2,\,-\alpha_3,-\abc} \\
& = & \Qf{-1}\Pf{\lam^2+4\abc} \tensor\Qf{\alpha_1,\,\alpha_2,\,\alpha_3,\abc} \\
& = & \Qf{-\abc}\Pf{\lam^2+4\abc} \tensor\Pf{-\alpha_1\alpha_2,-\alpha_1\alpha_3} \\
& = & \Pf{\lam^2+4\abc} \tensor\Pf{-\alpha_1\alpha_2,-\alpha_1\alpha_3}.
\end{eqnarray*}
where the last step follows from $-\abc$ being a value of $\Pf{\lam^2+4\abc}$. 
\end{proof}

\begin{cor}
$\Caltf(\varphi)$ splits if and only if $\varphi$ is isotropic.
\end{cor}

\subsection{Residues}

Faddeev's exact sequence (see \cite[Section~6.4]{GS}) shows that $\mu(\varphi)$ is determined by its residues. Since $\deg(\lam^2+4\abc)$ is even, the only nonzero residues are obtained from divisors of $\lam^2+4\abc$:
\begin{rem}
The nonzero residues of $\mu(\varphi)$ are:
\begin{enumerate}
\item The associative Clifford algebra $\Cass(\varphi)$ at $\lam^2+4\abc$ if $\disc(\varphi) \not \equiv 1$,
\item $(\alpha_i,\alpha_j)$ at $\lam \pm \sqrt{-4\abc}$, if $\disc(\varphi) \equiv 1$.
\end{enumerate}
\end{rem}

\begin{cor}
The associative Clifford algebra $\Cass(\varphi)$ (up to isomorphism over $F$) and the localized alternative Clifford algebra $\Caltf(\varphi)$ (up to isomorphism over $F(\lam)$) determine each other.
\end{cor}

\bibliographystyle{abbrv}

\end{document}